\documentclass[11pt,a4paper,english]{article}
\pdfoutput=1

\usepackage[utf8]{inputenc}
\usepackage[english]{babel}
\usepackage[affil-it]{authblk}
\usepackage{graphicx}
\usepackage{bbm}
\usepackage{amsthm,amsmath,amssymb,amsfonts}
\usepackage{mathtools}
\usepackage{mathrsfs}
\usepackage[hyphens]{url}
\allowdisplaybreaks

\newtheorem{theorem}{Theorem}[section]
\newtheorem{corollary}[theorem]{Corollary}
\newtheorem{lemma}[theorem]{Lemma}
\newtheorem{proposition}[theorem]{Proposition}

\theoremstyle{definition}
\newtheorem{definition}[theorem]{Definition}

\theoremstyle{remark}
\newtheorem{remark}[theorem]{Remark} 
\newtheorem{example}[theorem]{Example}

\newcommand{\E}{\mathbb{E}}
\newcommand{\1}{\mathbbm{1}} 
\newcommand{\VaR}{\text{VaR}}
\newcommand{\AVaR}{\text{AVaR}}
\DeclareMathOperator*{\essinf}{ess\,inf}
\DeclareMathOperator*{\esssup}{ess\,sup}

\newcommand{\N}{\mathbb{N}_{0}}

\newcommand{\ud}{\mathrm{d}}

\newcommand{\PP}{\mathbb{P}}

\newcommand{\msc}[1]{\textbf{MSC2010 Classification:} #1.}
\newcommand{\jel}[1]{\textbf{JEL Classification:} #1.}
\newcommand{\keywords}[1]{\textbf{Key words:} #1.}

\usepackage{color}

\begin{document}
	
	\title{Markov risk mappings and risk-sensitive optimal prediction\thanks{The authors would like to thank the Isaac Newton Institute for Mathematical Sciences for support and hospitality during the programme `The Mathematics of Energy Systems' when work on this paper was undertaken. This work was partially supported by EPSRC grant numbers EP/R014604/1, EP/N013492/1 and EP/P002625/1. This work was supported by the Lloyd’s Register Foundation-Alan Turing Institute programme on Data-Centric Engineering under the LRF grant G0095.}
	}
	
	
	\author[]{Tomasz Kosmala\thanks{Corresponding author. Email: t.kosmala@qmul.ac.uk} }
	\author[]{Randall Martyr}
	\author[]{John Moriarty}
	
	\affil[]{School of Mathematical Sciences, Queen Mary University of London, 
		\\ Mile End Road, London E1 4NS, United Kingdom \\
		Tel.: +44 (0)20 7882 5440}
	
	\date{\today}

	\maketitle
	
	\begin{abstract}
		We formulate a probabilistic  Markov property in discrete time under a dynamic risk framework with minimal assumptions. This is useful for recursive solutions to risk-sensitive versions of dynamic optimisation problems such as optimal prediction, where at each stage the recursion depends on the whole future. The property holds for standard measures of risk used in practice, and is formulated in several equivalent versions including a representation via acceptance sets, a strong version, and a dual representation.
		
		\vspace{+4pt}
		\vspace{+4pt}
		
		\msc{60G40, 91B08, 91B06, 90C40}
		\vspace{+4pt}
		
		\jel{C61, D81}
	\end{abstract}
	
	\section{Introduction}
	
	The Markov property is a main tool used in the dynamic evaluation of risk, for example in the solution of risk-sensitive optimisation problems. In this paper we present a probabilistic formulation of the Markov property under a risk framework with minimal assumptions, which we call {\em dynamic conditional risk mappings}, and give applications to optimal prediction, a class of risk-sensitive stochastic optimisation problems.
	
	To fix ideas, let $X = (X_t)_{t \in \N}$ be a Markov chain taking values in a measurable space $E$ and let $(\Omega,\mathcal{F},(\mathcal{F}_t)_{t \in \N},\mathbb{P}^x)$ be its canonical probability space, where $X_0=x$, $\mathbb{P}^x-$a.s. and $\N=\{0,1,\ldots\}$. Let $\varrho=((\rho_t^x)_{t \in \N})_{x \in E}$ be the family of conditional linear expectations given by  
	\begin{align}\label{linearcase}
		\rho^{x}_{t}(Z) = 
		\begin{cases}
			\mathbb{E}^{x}[Z], & t=0, \\ 
			\mathbb{E}^{x}[Z \vert \mathcal{F}_{t}], & t \geq 1,
		\end{cases} 
	\end{align}
	where $Z$ is an arbitrary bounded random variable depending on the whole sample path (that is, $Z$ is measurable with respect to $(\Omega,\mathcal{F})$). Then $\varrho$ is Markovian in the sense that 
	\begin{equation}\label{Markov_prop_intro}
		\rho^{x}_{t}(Z \circ \theta_{t}) = \rho^{X_{t}}_0(Z) \;\;\mathbb{P}^{x}\text{-a.s. for each $t \in \N$},
	\end{equation}
	where $\theta_t$ is the shift operator, and we would like to generalise this property to an appropriately large class of nonlinear (that is, risk-sensitive) families $\varrho$.
	
	A number of settings have been given for the Markov property under dynamic risk frameworks. Broadly they are formulated either on functions of the state of the Markov process (that is, analytically), or on the canonical probability space (that is, probabilistically). 
	In the linear case the probabilistic and analytic formulations are equivalent, and below we obtain sufficient conditions on $\varrho$ for their equivalence (Proposition \ref{pro_general_Markov_property}).
	
	Analytic formulations, which are often based on so-called transition risk mappings (cf.\ Definition \ref{def:amp}), are useful in recursive solution techniques which evaluate risk only one step ahead, taking $Z=f(X_{t+1})$ in \eqref{linearcase}. Since probabilistic formulations apply to the whole path of $X$, they are useful for recursions which directly evaluate risk multiple steps ahead, taking $Z=f(X_{t+1}, X_{t+2},...)$. In optimal prediction problems, for example, the evaluation of risk depends on the evolution of the process after a user-selected stopping time: for instance, the problem of stopping as close as possible to the ultimate maximum of a time-homogeneous Markov chain $X$ taking values in $E = \mathbb{R}$ (cf.\ \cite{Allaart,Yam_Yung_Zhou} in the case of linear expectation):
	\begin{equation}
		\label{sensitive_opimal_pred}
		V_\text{pred}^T(x) := \inf_{\tau \in \mathscr{T}_{[0,T]}} \rho^x_0(X_T^* - X_\tau),
	\end{equation}
	where $T \in \N$, $X_T^* := \max_{0 \leq s \leq T} X_s$, and $\mathscr{T}_{[0,T]}$ is the set of stopping times taking values in $\{0,1,\ldots,T\}$ (see also \cite{Toit_Peskir,Pedersen} for work in continuous time). 
	In the aforementioned studies, explicit solutions have been obtained for this problem by applying the probabilistic Markov property to represent the objective as a function of $\tau$ and $X_\tau$, obtaining a function $F$ such that
	\begin{equation*}
		\E^x[ f(X_T^* - X_\tau)] = \E^x[ F(\tau,X_\tau)],
	\end{equation*}
	see e.g. page 1077 in \cite{Allaart}.
	The probabilistic Markov property, which is satisfied by the commonly used entropic, mean semi-deviation, $\VaR$, $\AVaR$ and worst-case risk mappings (see Section \ref{sec:examples}), is applied in Section \ref{sec:op} to solve \eqref{sensitive_opimal_pred} recursively.
	
	The evaluation of risk for random variables via so-called sublinear (and therefore convex) functionals goes back to \cite{Lebedev1992,Lebedev1993}, where results including dual representations are obtained in both the static and conditional settings. To enable dynamic programming for risk-sensitive Markov decision processes, dynamic risk-sensitive frameworks have also been proposed using analytic formulations of the Markov property under the assumption of time consistency. In \cite{Ruszczynski} a dynamic setting is introduced in which risk-sensitive Markov decision processes are studied in both finite and infinite time horizon. Also with infinite time horizon, the average risk of controlled Markov processes is studied in \cite{Shen2013} while \cite{Cavus_Ruszczynski} address the undiscounted total risk of transient controlled Markov processes.
	In \cite{Fan2018a, Fan2018b} a structure for dynamic risk measures is introduced based on a stronger concept of stochastic conditional time consistency.
	Utility-based (also known as certainty equivalent) frameworks are special cases using the Markov property under linear expectation, see for example \cite{Bauerle2014,Bauerle2017b}.
	Other frameworks are presented in \cite{Bartl2019} using analytic sets and in \cite{Pichler2019} using the Kusuoka representation. 
	
	While the dynamic risk-sensitive frameworks above involve a reference probability measure, analytic settings of the Markov property also exist in risk-sensitive frameworks without such a measure. When the state space $E$ is finite, these frameworks include Markov chains under imprecise expectations, which are related to sensitivity analyses under a set of possible transition probabilities for the Markov process $(X_{t})_{t \in \N}$, see for example \cite{deCooman_Hermans_Quaeghebeur_2009,Hartfiel,Krak_Bock_Siebes}. More generally they include nonlinear expectations which, in \cite{peng2005nonlinear} and \cite{Nendel2021}, are related to finite-dimensional properties of so-called nonlinear Markov chains. As in the present paper, in \cite{Denk_Kupper_Nendel} the framework is related to the infinite dimensional path space of Markov processes, although convexity of the nonlinear expectation is then assumed. 
	Also without a reference measure, the risk forms of \cite{Dentcheva_Ruszczynski_forms} have been applied to the optimisation of partially observable two-stage systems. 
	
	The general study of dynamic conditional risk mappings can also be approached via backward stochastic differential or difference equations, see \cite{Cohen_Elliott2008,Cohen_Elliott2010}.
	In contrast to the latter setup our risk mappings do not assume time consistency. In the other direction, in \cite{martyr_moriarty_perninge_2022} reflected backward stochastic difference equations are derived from dynamic conditional risk mappings, in the study of non-Markovian optimal switching problems.
	
	In the present work we assume a reference measure and make minimal further assumptions. Time consistency is not assumed, making our formulation applicable to risk mappings including mean semi-deviation and average value at risk (cf.\ Section \ref{sec:examples}). In the time-consistent case we make the connection to analytic formulations, and provide a recursive solution to the optimal prediction problem.
	
	For convex risk mappings we characterise the Markov property in terms of the dual representation (see for example \cite{Artzner_Delbaen_Eber_Heath,Delbaen,Detlefsen_Scandolo,Frittelli_Rosazza_Gianin,Lebedev1992,Lebedev1993}).
	More precisely, we show that a Markovian convex risk mapping can be characterised as a supremum over penalised linear expectations with respect to certain transition kernels, extending the dual representation of transition risk mappings beyond the coherent case studied in \cite{Ruszczynski}. We also obtain sufficient conditions under which the latter structure implies the probabilistic Markov property.
	
	The paper is structured as follows. Section \ref{sec:Markov-Dynamic-Conditional-Risk-Mappings} provides the probabilistic framework, together with equivalences between versions of the Markov property, and a representation in terms of acceptance sets. Section \ref{sec:examples} gives examples and Section \ref{sec:dual} addresses the dual representation, while applications to optimisation problems are given in Section \ref{sec:Risk-Averse-Optimal-Stopping}.

	\section{A probabilistic Markov property for risk mappings}\label{sec:Markov-Dynamic-Conditional-Risk-Mappings}
	
	After presenting the setup and briefly recalling necessary definitions (Section \ref{sec:mkc}), in Sections \ref{sec:mpos} and \ref{sec:ef} we provide our novel probabilistic setting for the Markov property and establish equivalent forms. The Markov property in terms of acceptance sets is studied in Section \ref{sec:mpas}.
	
	\subsection{Setup and notation}\label{sec:mkc}
	
	Suppose we have an $E$-valued time-homogeneous Markov process $(X_{t})_{t \in \N}$ with respect to the filtered probability space $(\Omega,\mathcal{F},\mathbb{F},\mathbb{P})$, where:
	\begin{itemize}
		\item[\textbullet] $E$ is a Polish space equipped with its Borel $\sigma$-algebra $\mathcal{E}$,
		\item[\textbullet] $\N = \{0,1,2,\ldots\}$ is the discrete time parameter set,
		\item[\textbullet] $\Omega$ is the canonical space of trajectories $\Omega = E^{\N}$,
		\item[\textbullet] $X$ is the coordinate mapping, $X_{t}(\omega) = \omega(t)$ for $\omega \in \Omega$ and $t \in \N$,
		\item[\textbullet]  $\mathbb{F} = (\mathcal{F}_{t})_{t \in \N}$ with $\mathcal{F}_{t} = \sigma(\{X_{s} \colon s \le t\})$ the natural filtration generated by $X$ and $\mathcal{F} = \sigma(\bigcup_{t \in \N} \mathcal{F}_{t})$.
	\end{itemize}
	Let $\mathscr{P}(\mathcal{F})$ denote the set of probability measures on $(\Omega,\mathcal{F})$. Unless otherwise specified, all inequalities between random variables will be interpreted in the almost sure sense with respect to the appropriate probability measure. We write $\mathscr{T}$ for the set of finite-valued stopping times and $\mathscr{T}_{[t,T]}$ for the set of stopping times taking values in $\{t,t+1,\ldots,T\}$. We denote by $b\mathcal{F}$ the space of bounded random variables on $(\Omega,\mathcal{F})$ and similarly for other $\sigma$-algebras. 
	It will also be convenient to define $\mathcal{F}_{t,\infty} = \sigma(X_{s} \colon s \ge t)$ and $\mathcal{F}_{t,t} = \sigma(X_t)$. 
	
	In the above setup the following objects exist:
	\begin{itemize} 
		\item[\textbullet] The law $\mu^{X_{0}}$ of $X_0$ under $\mathbb{P}$ and a family of probability measures defined by the measurable mapping $x \mapsto \mathbb{P}^x$ from $E$ to $\mathscr{P}(\mathcal{F})$, which is a disintegration of $\mathbb{P}$ with respect to $X_0$ (see \cite{DellacherieMeyer1978}, p.\ 78). To be precise, this family satisfies $\mathbb{P}^{x}(X_{0} = x) = 1$ and
		for every $F \in \mathcal{F}$ we have
		\begin{equation*}
			\mathbb{P}(F) = \int_{E}\mathbb{P}^{x}(F)\,\mu^{X_{0}}(\ud x).
		\end{equation*}
		\item[\textbullet] A time-homogeneous Markov transition kernel $q^{X} \colon \mathcal{E} \times E \to [0,1]$ such that for every $x \in E$ and $B \in \mathcal{E}$ we have $q^{X}(B \vert x) = \mathbb{P}^{x}\big(X_{1} \in B \big)$, 
		\item[\textbullet] Markov shift operators $\theta_{t} \colon \Omega \to \Omega$, $t \in \N$ such that $\theta_{0}(\omega) = \omega$, $\theta_{t} \circ \theta_{s} = \theta_{t+s}$ and $(X_{t} \circ \theta_{s})(\omega) = X_{t+s}(\omega)$ for each $\omega \in \Omega$ and $s,t \in \N$.
	\end{itemize}
	
	For $\tau \in \mathscr{T}$ define the random shift operator $\theta_{\tau}$ by
	\[
	\begin{split}
		\theta_{\tau}(\omega) & = \theta_{\tau(\omega)}(\omega), \\
		& = \theta_{t}(\omega)\;\; \text{on}\;\; \{\tau(\omega) = t\}.
	\end{split}
	\]
	
	We recall the definitions of risk mapping and conditional risk mapping (which are interchangeable via the mapping $Z \mapsto \rho(-Z)$ with the {\it monetary conditional risk measures} of \cite{Follmer_Schied}, Def.\ 11.1):
	
	\begin{definition}[Risk mapping]\label{Definition:Risk-Mapping}
		A risk mapping on the probability space $(\Omega,\mathcal{F},\mathbb{P}^x)$ is a function $\rho^x \colon b\mathcal{F} \to \mathbb{R}$ satisfying
		\begin{description}
			\item[\it Normalisation:] $\rho^x(0) = 0$,
			\item[\it Translation invariance:] $\forall\; Z \in b\mathcal{F}$ and $c \in \mathbb{R}$ we have $\rho^x(Z + c) = c + \rho^x(Z)$,
			\item[\it Monotonicity:] $\forall\; Z,Z' \in b\mathcal{F}$, we have $Z \le Z' \, \PP^x\text{-a.s.}	\implies \rho^x(Z) \le \rho^x(Z')$.
		\end{description}
	\end{definition}

	\begin{definition}[Conditional risk mapping]\label{Definition:Conditional-Risk-Mapping}
		A conditional risk mapping  on the probability space $(\Omega,\mathcal{F},\mathbb{P}^x)$ 
		with respect to the $\sigma$-algebra $\mathcal{F}_t \subseteq \mathcal{F}$ 
		is a function $\rho_t^x \colon b\mathcal{F} \to b\mathcal{F}_t$ satisfying:
		\begin{description}
			\item[\it Normalisation:] $\rho_t^x(0) = 0$ $\mathbb{P}^x$-a.s.,
			\item[\it Conditional translation invariance:] $\forall\; Z \in b\mathcal{F}$ and $Z' \in b\mathcal{F}_t$,
			\[
			\rho_t^x(Z + Z') = Z' + \rho_t^x(Z), \qquad \PP^x\text{-a.s.}
			\]
			\item[\it Monotonicity:] $\forall\; Z,Z' \in b\mathcal{F}$,
			\[
			Z \le Z' \, \PP^x\text{-a.s.}
			\implies \rho_t^x(Z) \le \rho_t^x(Z') \,\PP^x\text{-a.s.}
			\]
		\end{description}
	\end{definition}
	
	Conditional risk mappings also satisfy the following property (cf.\ \cite{Cheridito_Delbaen_Kupper_dyn}, Prop.\ 3.3 and \cite{Follmer_Schied}, Ex.\ 11.1.2):
	\begin{description}
		\item[\it Conditional locality:] for every $Z$ and $Z'$ in $b\mathcal{F}$ and $A \in \mathcal{F}_t$, we have $\PP^x$-a.s.
		\[
		\rho_t^x(\1_{A}Z + \1_{A^{c}}Z') = \1_{A}\rho_t^x(Z) + \1_{A^{c}}\rho_t^x(Z').
		\]
	\end{description}
	\begin{definition}[Dynamic conditional risk mapping]\label{Definition:Dynamic-Conditional-Risk-Mapping}
		For each $x \in E$ a dynamic conditional risk mapping on the filtered probability space $(\Omega,\mathcal{F},\mathbb{F},\mathbb{P}^x)$ is a sequence $(\rho_{t}^x)_{t \in \N}$ where 
		\begin{itemize}
			\item $\rho_0^x$ is a risk mapping,
			\item for each $t \geq 1$, $\rho_{t}^x$ is a conditional risk mapping on $(\Omega,\mathcal{F},\mathbb{P}^x)$ with respect to $\mathcal{F}_{t}$. 
		\end{itemize}
		We use the superscript $x$ in $(\rho_{t}^x)_{t \in \N}$ to indicate a dynamic conditional risk mapping on $(\Omega,\mathcal{F},\mathbb{F},\mathbb{P}^x)$.
	\end{definition}
	
	Note that the codomain of $\rho_0^x$ is $\mathbb{R}$ while, for each $t \geq 1$, the codomain of $\rho_t^x$ is $b{\mathcal{F}_t}$. This setup is motivated by the fact that any $\mathcal{F}_0$-measurable random variable is $\mathbb{P}^x$-a.s.\ constant. For example, for each $x \in E$, the sequence $(\rho_t^x)_{t \in \N}$ given by \eqref{linearcase} is a dynamic conditional risk mapping.
	
	For a finite stopping time $\tau$ define
	\[
	\rho_{\tau} = \sum_{t \in \N}\1_{\{\tau = t\}}\rho_{t},
	\]
	noting that $\rho_{\tau} \colon b\mathcal{F} \to b\mathcal{F}_{\tau}$.
	
	In some results below we will assume continuity.
	
	\begin{definition}
		Let $t\in \mathbb{N}_0$, $x \in E$. We say that $\rho^x_{t}$ is continuous from below (resp.\ from above) if $\rho^x_{t}(Y_n) \to \rho^x_{t}(Y)$ $\PP^x$-a.s.\ for every increasing (resp.\ decreasing) sequence $(Y_n)_{n \in \N}$ in $b\mathcal{F}$ converging $\PP^x$-a.s.\ to $Y \in b\mathcal{F}$.
	\end{definition}
	Note that results for decreasing risk maps (e.g.\ in \cite{Follmer_Penner}) requiring continuity from above can be applied to increasing risk maps of Definitions \ref{Definition:Risk-Mapping}-\ref{Definition:Dynamic-Conditional-Risk-Mapping} if continuity from below is assumed.

	\subsection{Markov property}\label{sec:mpos}
	
	We begin with measurability with respect to the initial state of the Markov process, referring to this as regularity.
	
	\begin{definition}[Regularity]\label{defn:reg}
		A collection of risk mappings $(\rho^x)_{x \in E}$ is said to be {\em regular} if for all $Z \in b\mathcal{F}$ the map $x \mapsto \rho^x(Z)$ is bounded and measurable.
	\end{definition}
	
	\begin{definition}[Markov property]\label{Definition:Markov-Conditional-Risk-Mapping}
		The family $\varrho:=((\rho_t^x)_{t \in \N})_{x \in E}$ of dynamic conditional risk mappings satisfies the {\it Markov property} (for the chain $(X_t)_{t \in \N}$) if 
		\begin{enumerate}
			\item $(\rho_0^x)_{x \in E}$ is regular,
			\item for each $x \in E$, $Z \in b\mathcal{F}$ and $t \in \N$ we have
			\begin{equation}\label{eq:Markov-Risk-Property}
				\rho^{x}_{t}(Z \circ \theta_{t}) = \rho^{X_{t}}(Z) \;\;\mathbb{P}^{x}\text{-a.s.},
			\end{equation}
		\end{enumerate}
		where $\rho^{X_{t}}(Z)$ is interpreted as the random variable $\omega \mapsto \rho^{X_{t}(\omega)}(Z)$. 
	\end{definition}
	
	By construction, if $\varrho:=((\rho_t^x)_{t \in \N})_{x \in E}$ is a family of dynamic conditional risk mappings then $(\rho_0^x)_{x \in E}$ is a collection of risk mappings. For convenience we often write $\rho^x$ for $\rho^x_0$.
	
	In particular we have
	\begin{equation}\label{eq:rcrml}
		\rho^x(Z)=\rho^x(\1_{\{x\}}(X_0)Z), \qquad Z \in b\mathcal{F}, \; x \in E.
	\end{equation}
	
	Note that the linear conditional expectation \eqref{linearcase} satisfies this Markov property and corresponds to the risk-neutral case. Examples of $\rho$ which are risk sensitive are presented in Section \ref{sec:examples}.
	
	\begin{remark}
		Note that \eqref{eq:Markov-Risk-Property} could have been specified differently. For example, by relating all risk mappings $\rho_t^x$ to the same regular collection $(\rho^x)_{x \in E}$ in \eqref{eq:Markov-Risk-Property} we have imposed a time homogeneity on the measurement of risk. This is not essential, since taking a collection $\{\rho^{x,s}:  x \in E, s \in \N\}$  indexed also by time and specifying
		\begin{equation}\tag{3'}
			\rho^{x}_{t}(Z \circ \theta_{t}) = \rho^{X_{t},t}(Z) \;\;\mathbb{P}^{x}\text{-a.s.},
		\end{equation}
		the family of dynamic conditional risk mappings may be time-heterogeneous. 
	\end{remark}
	
	A regular collection $(\rho^x)_{x \in E}$ of risk mappings can also be used to construct a Markovian family $\varrho=((\rho_t^x)_{t \in \N})_{x \in E}$ satisfying Definition \ref{Definition:Markov-Conditional-Risk-Mapping}, as follows. We use the fact that any bounded $\mathcal{F}$-measurable random variable $Z$ can be represented as $Z=f(X_0,X_1,\ldots)$ for some measurable and bounded function $f \colon E^{\mathbb{N}_0} \to \mathbb{R}$, which follows by standard monotone class arguments, see \cite{Blumenthal_Getoor_Markov}, Prop.\ 0.2.7 or \cite{Cinlar2011}, Th.\ 2.4.4. As it is obtained without reference to any probability measure, the equality $Z=f(X_0,X_1,\ldots)$ holds for all (rather than almost all) $\omega \in \Omega$ and therefore the function $f$ is unique.
	
	\begin{proposition}
		\label{pro_associated_dynamic}
		Let $(\rho^x)_{x \in E}$ be regular. For each $x \in E$, $t \in \N$ and  $Z=f(X_0,X_1,\ldots)$ let
		\begin{equation*}
			\rho^x_t(Z)(\omega) \coloneqq \rho^{X_t(\omega)}(Z_t(X_0(\omega),\ldots,X_t(\omega))), \qquad \omega \in \Omega,
		\end{equation*}
		where
		\begin{equation*}
			Z_t(x_0,\ldots,x_t) \coloneqq f(x_0,\ldots, x_t, X_1,X_2,\ldots).
		\end{equation*}
		Then for each $x \in E$, $(\rho^x_t)_{t \in \N}$ is a dynamic conditional risk mapping and  the family $\varrho = ((\rho^x_t)_{t \in \N})_{x \in E}$ satisfies the Markov property.
	\end{proposition}

	\begin{proof} 
		Let $x \in E$, $t \in \N$ and $\omega \in \Omega$ be arbitrary. For compactness we will write $X_{0:t}(\omega)$ for $(X_0(\omega),\ldots, X_t(\omega)) \in E^{t+1}$. 
		Clearly $\rho_t^x$ is normalised, so we check conditional translation invariance and monotonicity.
		Taking $Z = f(X_0,X_1,\ldots) \in b\mathcal{F}$ and $W = g(X_0,\ldots,X_t) \in b\mathcal{F}_t$, by construction we have 
		\begin{align*}
			\rho_t^x(Z+W)(\omega) 
			&= \rho^{X_t(\omega)}(Z_t(X_0(\omega),\ldots, X_t(\omega)) + W_t(X_0(\omega),\ldots, X_t(\omega))) \\
			&= \rho^{X_t(\omega)}(Z_t(X_0(\omega),\ldots, X_t(\omega))) + W_t(X_0(\omega),\ldots, X_t(\omega)) \\
			&= \rho_t^x(Z)(\omega) + W(\omega).
		\end{align*}
		
		To check monotonicity let $Z = f(X_0,X_1,\ldots)$ and $Z' = f'(X_0,X_1,\ldots)$ be two bounded random variables such that $Z \leq Z'$ $\mathbb{P}^x$-a.s.
		We first show that $Z_t(X_{0:t}(\omega)) \leq Z'_t(X_{0:t}(\omega)), \PP^{X_t(\omega)}$-a.s. 
		Writing as usual $\PP^x(A \vert \mathcal{F}_t)$ for $\E^x[1_A \vert \mathcal{F}_t]$ for each $A \in \mathcal{F}$, and applying conditional locality and the Markov property, for almost all $\omega$ we have, with a slight abuse of notation
		\begin{align*}
			1 &= \PP^x(Z \leq Z' \vert \mathcal{F}_t)(\omega) \\
			& = \PP^x(f(X_{0:t}(\omega),X_{t+1}, \ldots) \leq f'(X_{0:t}(\omega), X_{t+1}, \ldots)\vert \mathcal{F}_t)(\omega) \\
			& = \PP^{X_t(\omega)}(f(X_{0:t}(\omega),X_1, \ldots) \leq f'(X_{0:t}(\omega),X_1, \ldots))
			\\
			& = \PP^{X_t(\omega)}(Z_t(X_{0:t}(\omega)) \leq Z'_t(X_{0:t}(\omega))).
		\end{align*}
		By the monotonicity of $\rho^{X_t(\omega)}$ we then have that $\mathbb{P}^x$-a.s., 
		\begin{align*}
			\rho_t^x(Z)(\omega) = \rho^{X_t(\omega)}\big(Z_t(X_{0:t}(\omega))\big) \leq \rho^{X_t(\omega)}\big(Z_t'(X_{0:t}(\omega))\big) = \rho_t^x(Z')(\omega).
		\end{align*}
		
		Lastly we verify the Markov property for the family $\varrho$.
		For $Z=f(X_0,X_1,\ldots)$ we have by construction and \eqref{eq:rcrml} that $\PP^x$-a.s.
		\begin{equation*}
			\rho_t^x(Z \circ \theta_t)(\omega) 
			= \rho^{X_t(\omega)}(f(X_t(\omega),X_1,X_2,\ldots))
			= \rho^{X_t(\omega)}(Z).
		\end{equation*} 
	\end{proof}

	\subsection{Equivalent forms of the Markov property}\label{sec:ef}
	
	Just as for the linear conditional expectation, the Markov property for risk mappings can be stated in several equivalent forms. We begin with the strong Markov property.
	
	\begin{proposition}[Strong Markov Property]
		\label{pro_strong_Markov}
		If $\varrho:=((\rho_t^x)_{t \in \N})_{x \in E}$ satisfies the Markov property  
		then for any stopping time $\tau \in \mathscr{T}$ and $Z \in b\mathcal{F}$
		we have
		\[
		\rho^{x}_{\tau}(Z \circ \theta_{\tau}) = \rho^{X_{\tau}}(Z)
		\;\;\mathbb{P}^{x}\text{-a.s.}.
		\]
	\end{proposition}
	\begin{proof}
		Using $\{\tau = t\} \in \mathcal{F}_{t}$, conditional locality and the Markov property
		we have $\mathbb{P}^{x}$-a.s.:
		\begin{align*}
			\rho^{x}_{\tau}(Z \circ \theta_{\tau}) 
			= \sum_{t=0}^{\infty}\1_{\{\tau = t\}}\rho^{x}_{t}(Z \circ \theta_{t})	
			= \sum_{t=0}^{\infty}\1_{\{\tau = t\}}\rho^{X_{t}}(Z) 
			= \rho^{X_{\tau}}(Z). 
		\end{align*}
	\end{proof}
	
	To make a connection to one-step Markov properties we will require time consistency:
	
	\begin{definition}\label{Definition:Time-Consistent}
		The family $\varrho:=((\rho_t^x)_{t \in \N})_{x \in E}$ is said to be time consistent if for all $Y,Z \in b\mathcal{F}$, $t \in \N$ and $x \in E$ we have
		\begin{equation*}
			\rho^x_{t+1}(Y) \leq \rho^x_{t+1}(Z) \; \PP^x\text{-a.s.} \implies \rho^x_t(Y) \leq \rho^x_t(Z) \; \PP^x\text{-a.s.}
		\end{equation*}
		We say that a regular collection $(\rho^x)_{x \in E}$ of risk mappings is {\it time consistent} if the associated Markovian dynamic conditional risk mapping (constructed in Proposition \ref{pro_associated_dynamic}) is time consistent.
	\end{definition}
	
	It is well known (see e.g. \cite{Acciaio_Penner}, Prop.\ 1.16) that we then have the following recursive relation: for every $x \in E$ and $0 \leq s\leq t$,
	\begin{equation*}
		\rho^x_s = \rho^x_s \circ \rho^x_t.
	\end{equation*}
	As noted in \cite{Follmer_Schied}, Exercise 11.2.2, this relation can be generalised to stopping times: for any bounded stopping times $\tau_1 \leq \tau_2$ one has
	\begin{equation}
		\label{recursivity_stopping_times}
		\rho_{\tau_1}^x = \rho_{\tau_1}^x \circ \rho_{\tau_2}^x.
	\end{equation}
	
	Since risk mappings are nonlinear in general, in the next proof we use a non-standard version of the Monotone Class Theorem (see Appendix \ref{sec:MCT})
	which, unlike \cite{Blumenthal_Getoor_Markov}, Th.\ 0.2.3, does not appeal to vector spaces.
	
	\begin{proposition}
		\label{pro_general_Markov_property}
		Let $\varrho:=((\rho_t^x)_{t \in \N})_{x \in E}$ be a family of dynamic conditional risk mappings such that  
		$(\rho^x)_{x \in E}$ is regular.
		Let each $\rho_t^x$ be continuous from above and below: that is, $\rho^x_{t}(Y_n) \to \rho^x_{t}(Y)$ a.s.\ for every $t\in \mathbb{N}_0$, $x \in E$ and monotone sequence $(Y_n)_{n \in \N}$ in $b\mathcal{F}$ converging to $Y \in b\mathcal{F}$. Then
		
		\begin{enumerate}
			\item[(i)] $\varrho$ is Markov iff for all $k\geq 0$ the $k$-step Markov property holds:
			\begin{equation}\label{eq:ksmp}
				\rho_t^x(f(X_{t+1},\ldots,X_{t+k})) = \rho^{X_t}(f(X_1,\ldots ,X_k)),
			\end{equation}
			$\mathbb{P}^x$-a.s.\ for all $t \in \N$, $x \in E$ and bounded measurable functions $f\colon E^{k} \to \mathbb{R}$.
			\item[(ii)] If the family $\varrho$ is time consistent, then $\varrho$ is Markov iff the one-step Markov property holds: for every $t\in \N$, $x \in E$ and bounded measurable function $f \colon E \to \mathbb{R}$ we have
			\begin{equation}
				\label{1_step_Markov}
				\rho_t^x(f(X_{t+1})) = \rho^{X_t}(f(X_1)), \qquad \mathbb{P}^x\text{-a.s.}.
			\end{equation}
		\end{enumerate}
	\end{proposition}
	
	\begin{remark}
		The assumptions of this proposition simplify in the case of convex risk mappings, for which continuity from above implies continuity from below (see proof of Corollary \ref{cor_general_Markov}).
	\end{remark}

	\begin{proof}
		For both claims (i) and (ii), the `only if' part is trivial and so it remains to establish the `if' part. We first prove this for claim (ii).
		
		Therefore let the family $\varrho$ be time consistent and suppose that the one-step Markov property \eqref{1_step_Markov} holds. We begin by showing, proceeding by induction on $k$, that the Markov property \eqref{eq:Markov-Risk-Property} holds for the class of simple functions -- that is, functions of the form 
		\begin{equation}\label{eq:general_markov_property_2}
			f(x_{t+1},\ldots,x_{t+k}) = \sum_{j=1}^{n}\alpha_{j} g_{j}(x_{t+1},\ldots,x_{t+k}), \quad t \in \N, n \ge 1, \alpha_{i} \in \mathbb{R}.
		\end{equation}	
		This is true for $k = 1$ since this is a special case of the one-step Markov property \eqref{1_step_Markov}. Suppose it is also true for some $k \ge 1$. We have
		\begin{align}\label{eq:general_markov_property_3}
			f(x_{t+1},\ldots,x_{t+k+1}) & = \sum_{j=1}^{n}\alpha_{j} g_{j}(x_{t+1},\ldots,x_{t+k+1}) \notag \\ 
			& = \sum_{j=1}^{n}\alpha_{j} \left(\prod_{i=1}^{k+1}\1_{A_{ij}}(x_{t+i})\right) \notag \\ 
			& = \sum_{j=1}^{n}\1_{A_{1j}}(x_{t+1})\left(\alpha_{j}\prod_{i=2}^{k+1}\1_{A_{ij}}(x_{t+i})\right).
		\end{align}
		By taking all possible intersections of the sets $A_{11},\ldots,A_{1n}$ and their complements, we can define $N \ge n$ mutually disjoint sets $\tilde{A}_{1},\ldots,\tilde{A}_{N}$ belonging to $\mathcal{E}$ such that
		\begin{align*}
			\sum_{j=1}^{n}\1_{A_{1j}}(x_{t+1})\left(\alpha_{j}\prod_{i=2}^{k+1}\1_{A_{ij}}(x_{t+i})\right) & = \sum_{\ell=1}^{N}\1_{\tilde{A}_{\ell}}(x_{t+1})\left(\sum_{j=1}^{n}\tilde{\alpha}_{\ell j}\prod_{i=2}^{k+1}\1_{A_{ij}}(x_{t+i})\right),
		\end{align*}
		where $\tilde{\alpha}_{\ell j} = \alpha_{j}$ if $A_{1j} \cap \tilde{A}_{\ell} \neq \emptyset$ and $\tilde{\alpha}_{\ell j} = 0$ otherwise. Therefore we can rewrite $f$ in \eqref{eq:general_markov_property_3} as
		\begin{equation}
			\label{first_component_taken_out}
			f(x_{t+1},\ldots,x_{t+k+1}) = \sum_{\ell=1}^{N}\1_{\tilde{A}_{\ell}}(x_{t+1})f_{\ell}(x_{t+2},\ldots,x_{t+k+1}),
		\end{equation}
		where the $\tilde{A}_{\ell}$ are mutually disjoint and each $f_{\ell}$ has the form \eqref{eq:general_markov_property_2}. Using the local property and time consistency for $\rho_t^x$, the induction hypothesis and the one-step Markov property we have
		\begin{align}\label{eq:general_markov_property_5}
			\rho_t^x(f(X_{t+1},\ldots, X_{t+k+1})) & = \rho_t^x\left(\sum_{\ell=1}^{N}\1_{\tilde{A}_{\ell}}(X_{t+1})f_{\ell}(X_{t+2},\ldots,X_{t+k+1})\right) \notag \\
			& =  \rho_t^x\left(\rho_{t+1}^x\left( \sum_{\ell=1}^{N}\1_{\tilde{A}_{\ell}}(X_{t+1})f_{\ell}(X_{t+2},\ldots,X_{t+k+1})\right)\right) \notag \\
			& =  \rho_t^x\left(\sum_{\ell=1}^{N}\1_{\tilde{A}_{\ell}}(X_{t+1}) \rho_{t+1}^x\left( f_{\ell}(X_{t+2},\ldots,X_{t+k+1})\right)\right) \notag \\
			& = \rho_t^x\left(\sum_{\ell=1}^{N}\1_{\tilde{A}_{\ell}}(X_{t+1})\rho^{X_{t+1}}(f_{\ell}(X_{1},\ldots,X_{k}))\right) \notag \\
			& = \rho^{X_{t}}\left( \sum_{\ell=1}^{N}\1_{\tilde{A}_{\ell}}(X_{1}) \rho^{X_{1}}(f_{\ell}(X_{1},\ldots,X_{k}))\right).
		\end{align}
		Note that for every realisation $x_{t}$ of $X_{t}(\omega)$ we have that almost surely under $\mathbb{P}^{x_t}$,
		
		\begin{align}\label{eq:general_markov_property_6}
			\sum_{\ell=1}^{N}\1_{\tilde{A}_{\ell}}(X_{1})\rho^{X_1}\left(f_{\ell}(X_{1},\ldots,X_{k})\right) 
			& = \sum_{\ell=1}^{N}\1_{\tilde{A}_{\ell}}(X_{1})\rho_1^{x_t}\left(f_{\ell}(X_{2},\ldots,X_{k+1})\right) \notag \\
			& = \rho_1^{x_t}\left( \sum_{\ell=1}^{N}\1_{\tilde{A}_{\ell}}(X_{1}) f_{\ell}(X_{2},\ldots,X_{k+1})\right) \notag \\
		\end{align}
		Therefore, by \eqref{first_component_taken_out}--\eqref{eq:general_markov_property_6} and time-consistency we have for almost every $\omega \in \Omega$:
		\begin{align*}
			\rho_t^x(f(X_{t+1},\ldots, X_{t+k+1}))(\omega)
			&= \rho^{X_t(\omega)} \left( \rho_1^{X_t(\omega)}\left( \sum_{\ell=1}^{N}\1_{\tilde{A}_{\ell}}(X_{1}) f_{\ell}(X_{2},\ldots,X_{k+1})\right) \notag \right) \nonumber\\
			&= \rho^{X_t(\omega)}(f(X_1,\ldots, X_{k+1})), \label{eq:kmp}
		\end{align*}
		and, by induction, the Markov property \eqref{eq:Markov-Risk-Property} holds for all functions $f$ of the form \eqref{eq:general_markov_property_2}.
		\vskip0.1em
		
		Next we appeal to the monotone class theorem. Let $\mathscr{H}_0$ be the set of random variables having the form $Z=f(X_0,\ldots,X_{k})$ for some $k \in \mathbb{N}_0$ and some $f$ of the form \eqref{eq:general_markov_property_2}.
		Clearly $\mathscr{H}_0$ is closed under the operation of taking the pointwise minimum.
		Let 
		\begin{equation*}
			\mathscr{H} \coloneqq \{Z \in b\mathcal{F} \colon \rho_t^x(Z \circ \theta_t) = \rho^{X_t}(Z)\,\,\, \mathbb{P}^x\text{-a.s.\ for all } x \in E, t \in \N\}.
		\end{equation*}
		
		We show that $\mathscr{H}_0 \subset \mathscr{H}$. Suppose that $Z=f(X_0,\ldots, X_k) \in \mathscr{H}_0$. 
		Then by conditional locality and the fact that $Z\in \mathscr{H}_0$ we have that for each $\omega \in \Omega$,
		\begin{align*}
			\rho_t^x(Z \circ \theta_t)(\omega) & = \rho_t^x(f(X_t(\omega), X_{t+1}, \ldots, X_{t+k}))(\omega) \\ & = \rho^{X_t(\omega)}(f(X_t(\omega), X_1, \ldots, X_{k}))(\omega) = \rho^{X_t(\omega)}(f(X_0, X_1, \ldots, X_{k}))(\omega),
		\end{align*}
		i.e. $Z \in \mathscr{H}$.
		
		The space $\mathscr{H}$ is closed under monotone limits and Theorem \ref{th_from_Bogachev} 
		implies that $\mathscr{H}$ contains all bounded $\sigma(\mathscr{H}_0)$-measurable functions. Since $\sigma(\mathscr{H}_0) = \mathcal{F}$ we conclude that the Markov property \eqref{eq:Markov-Risk-Property} holds on $\mathscr{H} = b\mathcal{F}$, completing the proof of claim (ii).
		
		To prove claim (i), note that \eqref{eq:ksmp} applies directly to all functions $f$ of the form \eqref{eq:general_markov_property_2}, in which case time consistency does not need to be assumed. We then appeal to the monotone class theorem as we did for claim (ii).
	\end{proof}
	
	The following result shows that Markovian families of conditional risk mappings which are continuous from above and below can be represented using the canonical form given in Proposition \ref{pro_associated_dynamic}.
	
	\begin{proposition}\label{pro_extended_Markov_property}
		Let $\varrho \coloneqq ((\rho_t^x)_{t \in \N})_{x \in E}$ be a family of dynamic conditional risk mappings continuous from above and below. Then $\varrho$ satisfies the Markov property if and only if for all $x \in E$, $t \in \N$, $Z = f(X_{0},X_{1},\ldots) \in b\mathcal{F}$ and $\PP^{x}$-almost every $\omega \in \Omega$ we have
		\begin{equation*}\label{eq:past-present-future-Markov-property}
			\rho^{x}_{t}(Z)(\omega) = \rho^{X_{t}(\omega)}(Z_{t}(X_{0}(\omega),\ldots,X_{t}(\omega))),
		\end{equation*}
		where $Z_{t}(x_{0},\ldots,x_{t}) \coloneqq f(x_{0},\ldots,x_{t},X_{1},X_{2},\ldots)$.
	\end{proposition}
	
	The proof is omitted as it follows the path analogous to that of Proposition \ref{pro_general_Markov_property}, namely showing the claimed property for simple random variables and appealing to the monotone class theorem.
	
	\begin{remark}\label{rem_extended_Markov_property}
		\begin{enumerate}
			\item[(i)] If in addition to the hypotheses of Proposition \ref{pro_extended_Markov_property} the family $\varrho$ is time consistent, then we recover a version of the Markov property which is similar to that of \cite{Nendel2021}:
			\[
			\rho^{x}(Z) = \rho^{x}\Big(\rho^{X_{t}}(Z_{t}(X_{0},\ldots,X_{t}))\Big), \quad \forall x \in E, t \in \N,  Z \in b\mathcal{F},
			\]
			where $Z = f(X_{0},X_{1},\ldots)$ and $Z_{t}(x_{0},\ldots,x_{t}) \coloneqq f(x_{0},\ldots,x_{t},X_{1},X_{2},\ldots)$.
			Note that in Definition 1.2 in \cite{Nendel2021} random variables of the form $Z = f(X_{0},X_{1},\ldots, X_t,X_{t+s})$ are taken into account.
			\item[(ii)] In \cite{Denk_Kupper_Nendel} a Kolmogorov-type theorem is established for conditional risk mappings which, like Proposition \ref{pro_extended_Markov_property}, leads to a risk mapping on path space, and Example 5.3 of the latter paper explores the case of discrete-time Markov chains.
		\end{enumerate}
	\end{remark}
	
	\subsection{Markov property in terms of acceptance sets}
	\label{sec:mpas}
	
	Particularly in the context of mathematical finance, conditional risk mappings can be characterised by their {\em acceptance sets} $\mathcal{A}_t^x$  \cite{Acciaio_Penner}, Sec.\ 1.4.1 or \cite{Follmer_Schied}, Sec.\ 4.1, where 
	\begin{equation*}
		\mathcal{A}_t^x := \{ Y \in b\mathcal{F} : \rho_t^x(Y) \leq 0 \,\, \mathbb{P}^x\text{-a.s.}\},
	\end{equation*}
	and so for completeness we also formulate the Markov property in these terms. 
	First define another acceptance set, which will be useful in formulating the Markov property:
	\begin{equation*}
		\tilde{\mathcal{A}}_t^x = \{ Y\in b\mathcal{F} : \rho^{X_t}(Y) \leq 0 \,\, \mathbb{P}^x\text{-a.s.}\}.
	\end{equation*}
	
	Note that for any $\mathcal{F}_{t,\infty}$-measurable random variable $Y=\hat{Y} \circ \theta_t$ with the representation $Y = f(X_t,X_{t+1},\ldots)$ one can define $Y\circ \theta_{-t} := \hat{Y} = f(X_0,X_1,\ldots)$.

	\begin{lemma}\label{lem:equivalence_acceptance_sets} 
		The family $\varrho:=((\rho_t^x)_{t \in \N})_{x \in E}$ is Markov if and only if $\rho_t^x \colon b\mathcal{F}_{t,\infty} \to b\mathcal{F}_{t,t}$ and for each $x \in E$, $t \in \N$ and $Z \in b\mathcal{F}$ we have
		\begin{equation}
			\label{equivalence_acceptance_sets}
			Z\circ \theta_t \in \mathcal{A}_t^x \iff Z \in \tilde{\mathcal{A}}_t^x.
		\end{equation}
	\end{lemma}
	
	\begin{proof}
		Necessity is obvious. Conversely, suppose that $\rho_t^x \colon b\mathcal{F}_{t,\infty} \to b\mathcal{F}_{t,t}$ and that the equivalence \eqref{equivalence_acceptance_sets} holds. Fix $Z \in b\mathcal{F}$, $t \in \N$ and $x \in E$.
		
		Step 1. We show that
		\begin{equation}
			\label{equality_for_rho_X_t}
			\rho^{X_t}(Z) = \essinf \{ Y=g(X_t) \in b\mathcal{F}_{t,t} \colon Z - Y \circ \theta_{-t} \in \tilde{\mathcal{A}}_t^{x} \}, \qquad \mathbb{P}^{x}\text{-a.s..}
		\end{equation}
		Proof of `$\geq$'. Let $g(y) := \rho^y(Z)$ for $y\in E$. Then $\mathbb{P}^y$-a.s. we have 
		$g(X_t)\circ \theta_{-t} = \rho^{X_0}(Z) = \rho^y(Z)$.
		Since $\rho^y(Z-g(X_t)\circ\theta_{-t}) = \rho^y(Z-\rho^y(Z)) = 0$ we have $\rho^{X_t}(Z-g(X_t)\circ\theta_{-t}) =0$, implying $Z - g(X_t)\circ \theta_{-t} \in \tilde{\mathcal{A}}_t^{x}$.
		
		Proof of `$\leq$'. Let $Y=g(X_t)$ belong to the set on the right-hand side of \eqref{equality_for_rho_X_t}. Then for $\Omega_1 = \{\omega \in \Omega: \rho^{X_t}(Z - Y \circ \theta_{-t}) \leq 0\}$ we have $\mathbb{P}^{x}(\Omega_1) = 1$. Let $\Omega_0 = \{\omega \in \Omega: \rho^{X_t}(Z)>Y\}$. 
		To show that $\Omega_0 \subset \Omega_1^c$, let $\omega \in \Omega_0$ and $x_t := X_t(\omega) = \omega(t)$. 
		Since $\omega \in \Omega_0$, we have that  $\rho^{X_t(\omega)}(Z) > Y(\omega)$, which is equivalent to $\rho^{x_t}(Z)>g(x_t)$.
		Then
		\begin{multline*}
			\rho^{X_t(\omega)}(Z-Y\circ \theta_{-t}) 
			= \rho^{x_t}(\1_{x_t}(X_0)(Z-Y\circ \theta_{-t})) \\
			= \rho^{x_t}(Z-g(x_t)) 
			= \rho^{x_t}(Z) - g(x_t)
			> 0,
		\end{multline*}
		i.e.\ $\omega \in \Omega_1^c$.
		Since $\mathbb{P}^x(\Omega_1^c) = 0$, it follows that $\mathbb{P}^x(\Omega_0) = 0$, which finishes the proof of the claim of Step 1.
		
		Step 2. To finish the proof, note from \cite{Acciaio_Penner}, Prop.\ 1.2 (modulo a minus sign which appears because \cite{Acciaio_Penner} considers decreasing risk mappings) that
		\begin{align*}
			\rho_t^x(Z \circ \theta_t) 
			& = \essinf \{ Y \in b\mathcal{F}_t \colon Z\circ \theta_t - Y \in \mathcal{A}_t^x \} \\
			& \leq \essinf \{ Y \in b\mathcal{F}_{t,t} \colon  Z \circ \theta_t - Y \in \mathcal{A}_t^x \} \leq \rho_t^x(Z \circ \theta_t),
		\end{align*}
		where the last inequality follows from the fact that $Y = \rho_t^x(Z \circ \theta_t) \in b\mathcal{F}_{t,t}$  and $Z \circ \theta_t - \rho_t^x(Z \circ \theta_t) \in \mathcal{A}_t^x$.
		Since for $Y \in b\mathcal{F}_{t,t}$ we have from \eqref{equivalence_acceptance_sets} that
		\[
		Z \circ \theta_t - Y \in \mathcal{A}_t^x \iff Z - Y \circ \theta_{-t} \in \tilde{\mathcal{A}}_t^x,
		\]
		Step 1 completes the proof. 
	\end{proof}
	
	\begin{remark}
		\label{rem_X_t_measurable}
		The above lemma implies in particular that for a Markovian risk map, if $Z \in b\mathcal{F}_{t,\infty}$, then $\rho_t^x(Z)$ is $\sigma(X_t)$-measurable.
	\end{remark}

	\section{Examples}\label{sec:examples}
	
	In this section we provide examples of Markovian families of dynamic conditional risk mappings. Note that the entropic and worst case risk mappings are time consistent (see \cite{Detlefsen_Scandolo}, Prop.\ 6 and \cite{Barron_Cardaliaguet_Jensen}, Th.\ 2.8(b)(ii) respectively), while the mean semi-deviation risk mapping and average value at risk are not (\cite{Follmer_Schied}, Ex.\ 11.13, \cite{Artzner_Delbaen_Eber_Heath_Ku}, p.\ 20-21). Below we take $Z \in b\mathcal{F}, t \in \N, x \in E$.

	\subsection{Composite risk mappings}\label{sec:crm}
	Let $K \in \N$ and for $k = 0,\ldots,K$ let $g_{k} \colon \mathbb{R}^{m_{k}} \times E \to \mathbb{R}$ be measurable functions bounded on compact sets such that $m_{0} = 1$ and $m_k = 2$ for $k\geq 1$, with the map $x \mapsto g_{k}(r_{k},x)$ bounded on $E$ for every $r_{k} \in \mathbb{R}^{m_{k}}$. Assume also that for each $x \in E$, the sequence $(\rho_t^x)_{t \in \N}$ is a dynamic conditional risk mapping, where $\rho_0^{x}(Z) = R_{K}^{x}(Z)$ and 
	$\rho^{x}_{t}(Z) = R_{K}^{x}(Z \vert \mathcal{F}_{t})$ for $t \geq 1$, with
	\begin{equation}\label{eq:Composite-RCRM}
		R_{k}^{x}(Z) = \begin{cases}
			\mathbb{E}^{x}\big[g_{0}(Z,X_{0})\big], & \text{if}\;\;k =0,\\
			\mathbb{E}^{x}\Big[g_{k}\big(Z,R_{k-1}^{X_{0}}(Z),X_{0}\big)\Big], &\text{if}\;\; 1 \le k \le K,
		\end{cases}
	\end{equation}
	\begin{equation}\label{eq:Composite-DCRM}
		R_{k}^{x}(Z \vert \mathcal{F}_{t}) = \begin{cases}
			\mathbb{E}^{x}\big[g_{0}(Z,X_{t}) \big\vert \mathcal{F}_{t}\big], & \text{if}\;\;k =0,\\
			\mathbb{E}^{x}\Big[g_{k}\big(Z,R_{k-1}^{x}(Z \vert \mathcal{F}_{t}),X_{t}\big) \big\vert \mathcal{F}_{t}\Big], &\text{if}\;\; k \ge 1.
		\end{cases}
	\end{equation}
	This family clearly includes the linear expectation ($K = 0$, $g_{0}(z,x) = z$) and its statistical estimation properties are studied in \cite{Dentcheva2017}.
	
	\begin{lemma}\label{Lemma:Composite-Dynamic-Conditional-Risk-Map-Markov-Property}
		The family of dynamic conditional risk mappings $\varrho =((\rho_t^x)_{t \in \N})_{x \in E}$ defined through \eqref{eq:Composite-RCRM} and \eqref{eq:Composite-DCRM} is Markovian.
	\end{lemma}
	
	\begin{proof}
		The Markov property holds at $k = 0$ since $\PP^x$-a.s.
		\begin{align*}
			R_{0}^{x}(Z \circ \theta_{t} \vert \mathcal{F}_{t}) & = \mathbb{E}^{x}\big[g_{0}(Z \circ \theta_{t},X_{t}) \big\vert \mathcal{F}_{t}\big] \\
			& = \mathbb{E}^{x}\big[g_{0}(Z,X_{0}) \circ \theta_{t} \big\vert \mathcal{F}_{t}\big] \\
			& = \mathbb{E}^{X_{t}}\big[g_{0}(Z,X_{0})\big] = R_{0}^{X_{t}}(Z).
		\end{align*}
		Assuming that it holds at $k-1$, the Markov property also holds at $k$:
		\begin{align*}
			R_{k}^{x}(Z \circ \theta_{t} \vert \mathcal{F}_{t}) & = \mathbb{E}^{x}\Big[g_{k}\big(Z \circ \theta_{t},R_{k-1}^{x}(Z \circ \theta_{t} \vert \mathcal{F}_{t}),X_{t}\big) \big\vert \mathcal{F}_{t}\Big] \\
			& = \mathbb{E}^{x}\Big[g_{k}\big(Z,R_{k-1}^{X_{0}}(Z),X_{0}\big) \circ \theta_{t} \big\vert \mathcal{F}_{t}\Big] \\
			& = \mathbb{E}^{X_{t}}\Big[g_{k}\big(Z,R_{k-1}^{X_{0}}(Z),X_{0}\big) \Big] = R_{k}^{X_{t}}(Z) \qquad \PP^x\text{-a.s.}
		\end{align*} 
	\end{proof}

	\subsubsection{Entropic risk mapping} \label{sec:erm}
	
	The entropic risk mapping (a special case of a certainty equivalent risk mapping, see \cite{Follmer_Schied}, Def.\ 2.36. or \cite{Bauerle2017b}) is Markovian since it is recovered from \eqref{eq:Composite-DCRM} by taking $K = 1$, $g_{1}(z,r,x) = \frac{1}{\gamma(x)}\ln(r)$ (restricting the domain of $r \mapsto g_{1}(z,r,x)$ to $(0,\infty)$) and $g_{0}(z,x) = e^{\gamma(x)z}$ in \eqref{eq:Composite-DCRM}, where $\gamma \colon E \to (0,\infty)$ is measurable and bounded away from both 0 and $\infty$, giving
	\begin{equation*}	
		\rho^{x}_{t}(Z) = 
		\begin{cases}
			\frac{1}{\gamma(x)}\ln\left(\mathbb{E}^{x}\left[e^{\gamma(x) Z} \right]\right), & t=0, \\
			\frac{1}{\gamma(X_{t})}\ln\left(\mathbb{E}^{x}\left[e^{\gamma(X_{t}) Z} \big\vert \mathcal{F}_{t}\right]\right), & t \geq 1.
		\end{cases}
	\end{equation*}

	\subsubsection{Mean-semideviation risk mapping}\label{sec:msm}
	
	Similarly, the mean--semideviation risk mapping satisfies the Markov property since it is recovered from \eqref{eq:Composite-DCRM} by taking $K = 2$, $g_{2}(z,r,x) = z + \kappa(x) \, r^{\frac{1}{p}}$, $g_{1}(z,r,x) = ((z - r)^{+})^{p}$ and $g_{0}(z,x) = z$ in \eqref{eq:Composite-DCRM}, where $\kappa \colon E \to [0,1]$ is measurable and $p \ge 1$ is an integer, 
	giving
	\begin{equation*}
		\rho^{x}_{t}(Z) = 
		\begin{cases}
			\mathbb{E}^{x}[Z] + \kappa(x)\left(\mathbb{E}^{x}\left[\big(\left(Z - \mathbb{E}^{x}[Z]\right)^{+}\big)^{p} \right]\right)^{\frac{1}{p}}, & t=0, \\
			\mathbb{E}^{x}[Z \vert \mathcal{F}_{t}] + \kappa(X_{t})\left(\mathbb{E}^{x}\left[\big(\left(Z - \mathbb{E}^{x}[Z \vert \mathcal{F}_{t}]\right)^{+}\big)^{p} \big\vert \mathcal{F}_{t}\right]\right)^{\frac{1}{p}}, & t \geq 1.
		\end{cases}
	\end{equation*}
	
	\subsection{Worst-case risk mapping}\label{sec:wcrm}
	The worst-case risk mapping is given by the family
	
	\begin{equation}\label{eq:WC-Dynamic-Conditional-Risk-Map}
		\rho^{x}_{t}(Z) = 
		\begin{cases}
			\mathbb{P}^{x}-\esssup(Z), & t = 0, \\
			\mathbb{P}^{x}-\esssup\left(Z\,\vert\,\mathcal{F}_{t}\right), & t \geq 1.
		\end{cases}
	\end{equation}
	For $t \geq 1$ this is the $\mathcal{F}_{t}$-conditional $\mathbb{P}^{x}$-essential supremum of $Z$, that is, the smallest $\mathcal{F}_{t}$-measurable random variable dominating $Z$ almost surely with respect to $\mathbb{P}^{x}$ \cite{Barron_Cardaliaguet_Jensen}, Prop.\ 2.6.

	\begin{lemma}\label{Lemma:WC-Dynamic-Conditional-Risk-Map-Markov-Property}
		The family of dynamic conditional risk mappings given by \eqref{eq:WC-Dynamic-Conditional-Risk-Map} is Markovian. 
	\end{lemma}
	
	\begin{proof}
		Supposing first that $Z$ is non-negative, then using \cite{Barron_Cardaliaguet_Jensen}, Prop.\ 2.12 and the Markov property of the conditional expectation, we have $\mathbb{P}^{x}$-a.s.:
		\begin{align*}
			\rho^{x}_{t}(Z \circ \theta_{t}) & = \lim_{p \to \infty}\left(\mathbb{E}^{x}\big[(Z \circ \theta_{t})^{p}\,\vert\, \mathcal{F}_{t}\big]\right)^{\frac{1}{p}} \nonumber \\
			& = \lim_{p \to \infty}\left(\mathbb{E}^{x}\big[Z^{p} \circ \theta_{t} \,\vert\, \mathcal{F}_{t}\big]\right)^{\frac{1}{p}} \nonumber \\
			& = \lim_{p \to \infty}\left(\mathbb{E}^{X_{t}}[Z^{p}]\right)^{\frac{1}{p}} = \rho^{X_{t}}(Z),
		\end{align*}
		while the case $t=0$ establishes measurability in $x$. For general $Z \in b\mathcal{F}$ we first set $Z_{c} \coloneqq Z + c$ with $c = \sup_{\omega}\vert Z(\omega) \vert$, then use translation invariance with respect to constants (see \cite{Barron_Cardaliaguet_Jensen}, Prop.\ 2.1),
		\[
		\rho^{x}_{t}(Z \circ \theta_{t}) = \rho^{x}_{t}(Z_{c} \circ \theta_{t}) - c = \rho^{X_{t}}(Z_{c}) - c = \rho^{X_{t}}(Z),
		\]
		completing the proof. 
	\end{proof}
	
	\subsection{Value at Risk}
	
	The value at risk may be defined by the family
	\begin{align}
		\rho^{x}_{t}(Z) = 
		\begin{cases}
			\VaR^{x}_{\lambda}(-Z), & t=0,\\
			\VaR^{x}_{\lambda}(-Z \vert \mathcal{F}_{t}), & t \geq 1, \label{eq:var-dynamic} 
		\end{cases}
	\end{align}
	where $\lambda \in (0,1)$,
	$$\VaR^{x}_{\lambda}(-Z) := \inf\{m \in \mathbb{R} \colon \mathbb{P}^{x}(m < Z) \leq \lambda \}$$
	and
	$$\VaR^{x}_{\lambda}(-Z \vert \mathcal{F}_{t}) := \mathbb{P}^x- \essinf\{m_{t} \in b\mathcal{F}_{t} \colon \mathbb{P}^{x}(m_{t} < Z \vert \mathcal{F}_{t}) \le \lambda\}$$
	for $t \geq 1$, see e.g.\ \cite{Follmer_Schied}, Sec.\ 4.4 \& Ex.\ 11.4.

	\begin{lemma}\label{lem:var-markov}
		The family of dynamic conditional risk mappings given by \eqref{eq:var-dynamic} is Markovian.
	\end{lemma}
	\begin{proof}
		We first show that $x \mapsto \rho^{x}(Z)$ is measurable. For $y \in \mathbb{R}$ we have
		\begin{align*}
			\{ x \in E : \rho^{x}(Z) < y \}
			&= \{ x \in E : \inf \{m \in \mathbb{R} \colon \mathbb{P}^{x}(m < Z) \le \lambda \} < y \} \\
			&= \{ x \in E : \exists m<y : \mathbb{P}^x(m < Z) \leq \lambda \} \\
			&= \{ x \in E : \exists m \in \mathbb{Q}, m<y \colon \mathbb{P}^x(m < Z) \leq \lambda \} \\
			&= \bigcup_{m\in (-\infty,y) \cap \mathbb{Q}} \{ x \in E : \mathbb{P}^x(m < Z) \leq \lambda \} \\
			&= \bigcup_{m\in (-\infty,y) \cap \mathbb{Q}} f_m^{-1}((-\infty,\lambda]),
		\end{align*}
		where for each $m\in \mathbb{R}$ the function $f_m \colon E \to \mathbb{R}$ is defined by $f_m(x) = \mathbb{P}^x(m < Z)$. 
		Note that each $f_m$ is measurable because 
		for every $A \in \mathcal{F}$ the mapping $E \ni x \mapsto \mathbb{P}^x(A)\in \mathbb{R}$ is measurable by the measurability of $E \ni x \mapsto \mathbb{P}^x\in \mathscr{P}(\mathcal{F})$.
		Thus $\{ x \in E \colon \rho^{x}(Z) < y \}$ is a measurable set.
		
		To show that $x \mapsto \rho^x(Z)$ is bounded note that, since $Z$ is bounded, there exists $M\in \mathbb{R}$ such that $\{\vert Z\vert <M\}=\Omega$. Thus $-M \leq \rho^x(Z) \leq M$ for all $x \in E$.
		
		Next we show that $\rho^{x}_{t}(Z \circ \theta_{t}) = \rho^{X_{t}}(Z)$ almost surely. Since $Z$ is bounded, let $m_{t}(X_{0},\ldots,X_{t}) \in b\mathcal{F}_{t}$ satisfy $\PP^x$-a.s.
		\[
		\mathbb{P}^{x}(m_{t}(X_{0},\ldots,X_{t}) < Z \circ \theta_{t} \vert \mathcal{F}_{t}) \le \lambda.
		\]
		Then for $\PP^x$-almost all $\omega \in \Omega$, by conditional locality and the Markov property we have 
		\begin{align*}
			\lambda & \geq \mathbb{P}^{x}(m_{t}(X_{0},\ldots,X_{t}) < Z \circ \theta_{t} \vert \mathcal{F}_{t})(\omega) \\
			& \quad = \mathbb{P}^{x}(m_{t}(X_{0}(\omega),\ldots,X_{t}(\omega)) < Z \circ \theta_{t} \vert \mathcal{F}_{t})(\omega) \\
			& \quad = \mathbb{P}^{X_t(\omega)}(m_{t}(X_{0}(\omega),\ldots,X_{t}(\omega)) < Z ),
		\end{align*}
		giving $m_{t}(\omega) \ge \rho^{X_{t}(\omega)}(Z)$. We conclude that $\rho^{x}_{t}(Z \circ \theta_{t}) \ge \rho^{X_{t}}(Z)$ almost surely under $\mathbb{P}^{x}$.
		
		Conversely we have by the Markov property that $\mathbb{P}^{x}\text{-a.s.}$,
		\begin{align*}
			\mathbb{P}^{x}(\rho^{X_{t}}(Z) < Z \circ \theta_{t}  \vert \mathcal{F}_{t})(\omega) &=  \mathbb{P}^{X_{t}(\omega)}(\rho^{X_{0}}(Z) < Z) \\
			&= \mathbb{P}^{X_{t}(\omega)}(\rho^{X_{t}(\omega)}(Z) < Z) \leq \lambda,
		\end{align*}
		and, since $\omega \mapsto \rho^{X_{t}(\omega)}(Z)$ is bounded and $\mathcal{F}_{t}$-measurable, we conclude that  $\rho^{x}_{t}(Z \circ \theta_{t}) \le \rho^{X_{t}}(Z)$. 
	\end{proof}

	\subsection{Average Value at Risk}
	
	For $\lambda \in (0,1)$ the average value at risk (see \cite{Acciaio_Penner}, Ex.\ 1.10) may be defined by the following family of dynamic conditional risk mappings:
	\begin{align}
		\rho^{x}_{t}(Z) = 
		\begin{cases}
			\text{AVaR}_{\lambda}^x(-Z), & t = 0, \\
			\text{AVaR}_{\lambda,t}^x(-Z), & t \geq 1, \label{eq:avar} 
		\end{cases}
	\end{align}
	where
	$$\text{AVaR}_{\lambda}^x(-Z) = 
	\E^{x}\left[\VaR^{x}_{\lambda}(-Z) +  \frac{1}{\lambda}(Z - \VaR^{x}_{\lambda}(-Z))^{+}\right]$$
	and
	$$\text{AVaR}_{\lambda,t}^x(-Z) =\E^{x}\left[\VaR^{x}_{\lambda}(-Z \vert \mathcal{F}_{t}) +  \frac{1}{\lambda}(Z - \VaR^{x}_{\lambda}(-Z \vert \mathcal{F}_{t}))^{+} \Big\vert \mathcal{F}_{t}\right]$$
	for $t \geq 1$.
	
	\begin{lemma}
		The family of dynamic conditional risk mappings given by \eqref{eq:avar} is Markovian.
	\end{lemma}
	
	\begin{proof}
		This follows from the Markov property for $\VaR^{x}_{\lambda}(\cdot \vert \mathcal{F}_{t})$ (Lemma \ref{lem:var-markov}), since $\PP^x$-a.s.
		\begin{align*}
			\text{AVaR}_{\lambda,t}^x(-Z \circ \theta_t)
			&= \VaR_\lambda^{X_t}(-Z) + \E^x \left[ \frac{1}{\lambda} (Z \circ \theta_t - \VaR_\lambda^{X_t}(-Z))^+ \big\vert \mathcal{F}_t \right] \\
			&= \VaR_\lambda^{X_t}(-Z) + \E^{X_t} \left[ \frac{1}{\lambda} (Z - \VaR_\lambda^{X_0}(-Z))^+ \right] \\
			&= \AVaR_\lambda^{X_t}(-Z).
		\end{align*}
	\end{proof}

	\section{Dual representation of convex Markovian risk mappings}
	\label{sec:dual}
	
	In this section we characterise the dual representation of convex Markovian risk mappings. Recalling from Section \ref{sec:mkc} that $(q^X(B\vert x) \colon B \in \mathcal{E}, x \in E)$ is the kernel associated to the Markov process $X$ under $\mathbb{P}$, we begin with the necessary definitions:

	\begin{definition}\label{def:amp}
		\begin{enumerate}
			\item[(i)] $\mathcal{R} \colon E \times b\mathcal{E} \to \mathbb{R}$ is a transition risk mapping (cf.\ \cite{Cavus_Ruszczynski,Fan2018a,Ruszczynski}) if:
			\begin{itemize}
				\item[\textbullet] for all $f \in b\mathcal{E}$, $x \mapsto \mathcal{R}(x,f)$ is bounded and measurable,
				\item[\textbullet] for all $x \in E$, $f \mapsto \mathcal{R}(x,f)$ satisfies
				\begin{itemize}
					\item {\it normalisation}: $\mathcal{R}(x,0) = 0$,
					\item {\it monotonicity}: $\mathcal{R}(x,f) \le \mathcal{R}(x,g)$ for all $f \le g$,
					\item {\it constant translation invariance}: $\mathcal{R}(x,f+c) = \mathcal{R}(x,f) + c$ for all constants $c$.
				\end{itemize}
			\end{itemize}
			
			\item[(ii)] A transition risk mapping is convex if for all $x \in E$, $f,g \in b\mathcal{E}$ and $\lambda \in [0,1]$ we have 
			\[
			\mathcal{R}(x,\lambda f + (1-\lambda) g) \leq \lambda  \mathcal{R}(x,f) + (1-\lambda) \mathcal{R}(x,g).
			\]
		\end{enumerate}
	\end{definition}
	
	Note that by Definitions \ref{Definition:Dynamic-Conditional-Risk-Mapping} and \ref{defn:reg}, a transition risk mapping can be derived from a regular collection of risk mappings $(\rho^x)_{x \in E}$ by writing
	\begin{equation}
		\label{mathcal_R}
		\mathcal{R}(x,f) := \rho^{x}(f(X_{1})) \qquad \text{for } f \in b\mathcal{E}.
	\end{equation}
	If the transition risk mapping $\mathcal{R}$ defined by \eqref{mathcal_R} is convex and continuous from below it has the following dual representation (cf.\ \cite{Follmer_Penner}, Th.\ 2.3):
	\begin{equation}
		\label{dual_2}
		\mathcal{R}(x,f)
		= \sup_{\substack{Q \in \mathscr{P}(\mathcal{E}), \\ Q \ll \mathbb{P}^x \circ X_1^{-1}}} \left( \int_E f(y) \, Q(\ud y) - \alpha^x(Q) \right),
	\end{equation}
	where the penalty functions $\alpha^x : \mathscr{P}(\mathcal{E}) \to \mathbb{R}$
	are defined by 
	\begin{equation*}
		\alpha^x(Q) 
		= \sup_{g \in b\mathcal{E}} \left( \E_Q[g] - \mathcal{R}(x,g) \right).
	\end{equation*}

	Letting $\mathcal{K}$ denote the set of kernels $q$ such that $q(\cdot\vert x)$ is absolutely continuous with respect to $q^X(\cdot\vert x)$ for every $x \in E$, we have the following proposition:
	
	\begin{proposition}
		\label{pro_dual_rep_1_step}
		Let $\varrho:=((\rho_t^x)_{t \in \N})_{x \in E}$ be a family of dynamic conditional risk mappings such that $(\rho^x)_{x \in E}$ is regular and each risk mapping $\rho^x$ is convex and continuous from below. Then $\varrho$ satisfies the one-step Markov property \eqref{1_step_Markov} if and only if for all $x \in E$ and $f \in b\mathcal{E}$ we have
		\begin{align}
			\label{dual_rep_kernels}
			\rho^x(f(X_{1})) &= \sup_{q \in \mathcal{K}} \left( \int_E f(y) \, q(\ud y \vert x) - \alpha^{x}(q(\cdot\vert x)) \right), \\
			\label{dual_rep_1_step}
			\rho_t^x(f(X_{t+1})) & = 
			\sup_{q \in \mathcal{K}} \left( \int_E f(y) \, q(\ud y \vert X_t) - \alpha^{X_t}(q(\cdot\vert X_t)) \right), & t = 1,2,\ldots.
		\end{align}
		$\PP^x$-almost surely.
	\end{proposition}
	\begin{proof}
		Using kernels, for all $x \in E$ and $f \in b\mathcal{E}$ the representation \eqref{dual_2} can be rewritten as \eqref{dual_rep_kernels}.
		Indeed, it is clear that the right-hand side of \eqref{dual_rep_kernels} is less than or equal to the right-hand side of \eqref{dual_2}. For the reverse inequality, simply note that for every given $x \in E$ and every $Q \in \mathscr{P}(\mathcal{E})$ such that $Q \ll \mathbb{P}^x \circ X_1^{-1}$, we can associate a kernel $q \in \mathcal{K}$ by setting $q( \cdot \vert x') = Q\1_{\{x\}}(x') + (\mathbb{P}^{x'} \circ X_1^{-1})(1-\1_{\{x\}}(x'))$. Then Equation \eqref{dual_rep_kernels} implies
		$$\sup_{q \in \mathcal{K}} \left( \int_E f(y) \, q(\ud y \vert X_t) - \alpha^{X_t}(q(\cdot \vert X_t)) \right) 
		= \rho^{X_t}(f(X_1)),$$
		which shows that \eqref{dual_rep_1_step} is equivalent to the one step Markov property \eqref{1_step_Markov}. 
	\end{proof}
	
	Note that in \eqref{dual_rep_kernels} we take the supremum (rather than essential supremum) over a potentially uncountable family of kernels. Therefore the regularity of the collection $(\rho^x)_{x \in E}$ follows from the assumptions of Proposition \ref{pro_dual_rep_1_step} rather than from \eqref{dual_rep_kernels}.
	
	\begin{corollary}
		\label{cor_general_Markov}
		Let $\varrho:=((\rho_t^x)_{t \in \N})_{x \in E}$ be a time-consistent family of dynamic conditional risk mappings such that each risk mapping $\rho^x$ is convex and continuous from above. Then $\varrho$ satisfies \eqref{dual_rep_kernels}--\eqref{dual_rep_1_step} (for all $x \in E$ and $f \in b\mathcal{E}$) iff the Markov property of Definition \ref{Definition:Markov-Conditional-Risk-Mapping} holds.
	\end{corollary}
	
	\begin{proof}
		Combining Lemma 4.21 and Theorem 4.22 in \cite{Follmer_Schied}, we see that a convex conditional risk mapping that is continuous from above is also continuous from below (recall the sign difference in our work). 
		The corollary is then an application of Proposition \ref{pro_general_Markov_property} to Proposition \ref{pro_dual_rep_1_step}. 
	\end{proof}
	
	The following example identifies a maximising kernel in \eqref{dual_rep_1_step} in the case of the entropic risk mapping of Section \ref{sec:erm}.
	
	\begin{example}[Entropic risk]
		\label{exa_entropic_risk_dual_rep}
		For $q \in \mathcal{K}$ let
		\begin{equation}\label{eq:emp}
			\alpha^x(q(\cdot\vert x)) = \frac{1}{\gamma(x)} \int_E \ln \frac{\ud q(\cdot \vert x)}{\ud  q^X(\cdot \vert x)}(y) \, q(\ud y \vert x).
		\end{equation}
		Fixing a bounded, measurable function $f\colon E \to \mathbb{R}$, define the kernel $q_{op}$ by
		\begin{equation}
			\label{optimal_kernel}
			\frac{\ud q_{op}(\cdot \vert x)}{\ud q^X(\cdot\vert x)}(y) 
			= \frac{e^{\gamma(x) f(y)}}{\int_E e^{\gamma(x) f(z)} \, q^X(\ud z\vert x) }.
		\end{equation}
		It is well known (\cite{Detlefsen_Scandolo}, Rem. 9) from the non-Markovian setting that for each $x \in E$ the function \eqref{eq:emp} is the minimal penalty corresponding to the entropic risk mapping on $b\mathcal{E}$ and that \eqref{optimal_kernel} defines a measure attaining the maximum in \eqref{dual_rep_1_step}.
		In order to show that $q_{op}$ defined in this way is indeed a kernel we show that for each $A \in \mathcal{E}$ the function $x \mapsto q_{op}(A\vert x)$ is measurable. 
		Indeed we have $q_{op}(A\vert x) = \int_E  \1_A(y) \frac{e^{\gamma(x)f(y)}}{\int_E e^{\gamma(x)f(z)} \, q^X(\ud z\vert x)} \, q^X(\ud y\vert x)$ and measurability follows since more generally, for any jointly measurable function $g\colon E \times E \to \mathbb{R}$ and any kernel $p$, the function $x \mapsto \int_E g(x,y) \, p(\ud y \vert x)$ is measurable.
	\end{example}
	
	\section{Applications}
	\label{sec:Risk-Averse-Optimal-Stopping}
	The probabilistic Markov property can provide a convenient tool to address, for example, optimal stopping problems with costs which are measurable only after the chosen stopping time. A first example is the case of exercise lag, where we seek
	\begin{equation*}
		L^T(x):=\inf_{\tau \in \mathscr{T}_{[0,T]}} \rho^x \left( \sum_{i=0}^{\tau-1} c(X_i) + g(X_{\sigma \circ \theta_\tau + \tau}) \right),
	\end{equation*}
	where functions $c, g:E \mapsto \mathbb{R}$ and a potentially unbounded stopping time $\sigma \in \mathscr{T}$ represent respectively an observation cost, exercise cost and exercise lag, and $\varrho=((\rho_t^x)_{t \in \N})_{x \in E}$ is a time-consistent Markovian family of dynamic risk mappings.
	The strong Markov property of Proposition \ref{pro_strong_Markov} then allows dynamic programming to be applied indirectly by first transforming the objective function. 
	Indeed it then follows by the recursive property \eqref{recursivity_stopping_times}, conditional locality, conditional translation invariance, the identity 
	$X_{\sigma \circ \theta_\tau + \tau} = X_\sigma \circ \theta_\tau$ and the strong Markov property that
	\begin{align*}
		L^T(x)&=\inf_{\tau \in \mathscr{T}_{[0,T]}}  \rho^x \left( \rho_\tau^x \left( \sum_{i=0}^{\tau-1} c(X_i) + g(X_{\sigma \circ \theta_\tau + \tau}) \right)\right) \nonumber\\
		&= \inf_{\tau \in \mathscr{T}_{[0,T]}}  \rho^x \left(\sum_{t=0}^T \rho_t^x \left( \1_{\{\tau=t \}} \sum_{i=0}^{t-1} c(X_i) + \1_{\{\tau=t \}} g(X_{\sigma \circ \theta_\tau + \tau}) \right) \right) \nonumber\\
		&= \inf_{\tau \in \mathscr{T}_{[0,T]}}  \rho^x \left( \sum_{i=0}^{\tau-1} c(X_i) + h(X_\tau)  \right),
	\end{align*}
	where $h(x) \coloneqq \rho^x(g(X_\sigma))$, and standard dynamic programming arguments can then be applied to obtain the Wald-Bellman equations
	\begin{equation*}
		\begin{cases}
			L^{0}(x) = h(x),& \\
			L^{m}(x) = h(x) \wedge \left(c(x) + \rho^{x}\big( L^{m-1}(X_{1})\big)\right),& m = 1,\ldots,T.
		\end{cases}
	\end{equation*}
	
	In the optimal prediction problem of the next section, use of the probabilistic Markov property enables dynamic programming to instead be applied directly.

	\subsection{Optimal prediction}
	\label{sec:op}
	
	Generalising \eqref{sensitive_opimal_pred}, let
	\begin{equation*}
		V_\text{pred}^T(x) := \inf_{\tau \in \mathscr{T}_{[0,T]}} \rho^x(g(X_T^* - X_\tau)),
	\end{equation*}
	where $x \in E = \mathbb{R}$, $\Omega = \mathbb{R}^{\mathbb{N}_0}$, $T \in \N$, $\varrho=((\rho_t^x)_{t \in \N})_{x \in E}$ is a Markovian family of dynamic conditional risk mappings, $X_T^* = \max_{0 \leq s \leq T} X_s$ and $g:\mathbb{R} \to \mathbb{R}$ is bounded and measurable.
	
	We extend this probability space to include the process' running maximum by letting $\tilde{\Omega} = (\mathbb{R} \times \mathbb{R})^{\N}$. On this space, we have the canonical process $(X_t(\tilde{\omega}),M_{t}(\tilde{\omega})) = (\tilde{\omega}^{1}(t),\tilde{\omega}^{2}(t)) = \tilde{\omega}(t)$.
	Setting $\tilde{\mathbb{F}} = (\tilde{\mathcal{F}}_{t})_{t \in \N}$ with $\tilde{\mathcal{F}}_{t} = \sigma(\{(X_{s},M_{s}) \colon s \le t\})$ and $\tilde{\mathcal{F}} = \sigma \left(\cup_{t} \tilde{\mathcal{F}}_{t}\right)$, there exists (see, for example, \cite{Cinlar2011}, Th.\ 4.4.18) a unique probability measure $\tilde{\mathbb{P}}^{x,m}$ on $(\tilde{\Omega},\tilde{\mathcal{F}})$ such that $(X,M)$ is a time-homogeneous Markov chain on $(\tilde{\Omega},\tilde{\mathcal{F}},\tilde{\mathbb{F}},\tilde{\mathbb{P}}^{x,m})$ with $\tilde{\mathbb{P}}^{x,m}(X_{0} = x, M_{0} = m) = 1$ and transition kernel $q^{X,M}$ satisfying 
	$q^{X,M}(\ud x',\ud m' \vert x, m) = \delta_{m\vee x'}(\ud m') \, q^{X}(\ud x' \vert x)$ for all $(x, m) \in \mathbb{R}^2$.
	Note that $\tilde{\mathbb{P}}^{x,m}(M_{n} = X_n^*\vee m) = 1$ and, in particular, for $m=x$ we have $\tilde{\mathbb{P}}^{x,x}(M_{n} = X_n^*) = 1$. 
	Recalling Proposition \ref{pro_associated_dynamic}, define a regular collection of risk mappings by
	\[
	\rho^{x,m}(f(X_0, M_0, X_1, M_1, \ldots)) := \rho^x(f(X_0, m, X_1, X_1^* \vee m, \ldots)),
	\]
	and let $((\rho_t^{x,m})_{t \in \N})_{x,m \in \mathbb{R}}$ be the associated Markovian family of dynamic conditional risk mappings.
	
	\begin{theorem} If $((\rho_t^x)_{t \in \N})_{x \in E}$ is time consistent then, for each bounded measurable function $g:\mathbb{R} \to \mathbb{R}$,
		the extended value function 
		\begin{align*}
			\tilde{V}^T(x,m) := \inf_{\tau \in \mathscr{T}_{[0,T]}} \rho^{x,m}(g(M_T-X_\tau))
		\end{align*}
		satisfies the following modified Wald-Bellman equations:
		\begin{align*}
			\tilde{V}^0(x,m) &= g(m-x), \\
			\tilde{V}^n(x,m) &= \rho^{x,m}(g(M_n-x)) \wedge \rho^{x,m}(\tilde{V}^{n-1}(X_1,M_1)).
		\end{align*}
		The optimal prediction problem \eqref{sensitive_opimal_pred} satisfies $V_{\text{pred}}^n(x) = \tilde{V}^n(x,x)$.
	\end{theorem}
	
	\begin{proof} 
		Set
		\begin{align*}
			S_T^T &= g(M_T-X_T), \\
			S_n^T &= \rho_n^{x,m}(g(M_T - X_n)) \wedge \rho_n^{x,m}(S_{n+1}^T). 
		\end{align*}
		Following the outline of \cite{Peskir2006}, Sec.\ 1.2, we may now proceed in five steps:
		
		\textbf{Step 1.} We show that for all 
		$n = 0,1,\ldots,T$ and $k = T-n, T-n-1, \ldots, 0$, we have
		$$S_k^{T-n} \circ \tilde{\theta}_n = S_{k+n}^T.$$
		One can easily see that the claim is true for $k=T-n$. Further, by the Markov property and backward induction, for all $n$ the random variable $S_n^T$ is $\sigma(M_n,X_n)$-measurable (cf.\ Remark \ref{rem_X_t_measurable}).
		All subsequent equalities hold $\tilde{\PP}^{x,m}$-almost surely.
		For any $Z = \hat Z \circ \tilde{\theta}_k \in b\tilde{\mathcal{F}}_{k,\infty}$ we have 
		\begin{align*}
			\rho^{x,m}_{k}(Z) \circ \tilde{\theta}_{n} & = \rho^{x,m}_{k}(\hat Z \circ \tilde{\theta}_k) \circ \tilde{\theta}_{n} =
			\rho^{X_k,M_k}(\hat Z) \circ \tilde{\theta}_{n} = \rho^{X_{k+n},M_{k+n}}(\hat Z) \\
			& = \rho^{x,m}_{k+n}(\hat Z \circ \tilde{\theta}_{k+n}) = \rho^{x,m}_{k+n}(Z \circ \tilde{\theta}_{n}).
		\end{align*}
		This and the induction hypothesis imply
		\begin{align*}
			S_k^{T-n} \circ \tilde{\theta}_n 
			&= \rho_k^{x,m}(g(M_{T-n} - X_k)) \circ \tilde{\theta}_n \wedge \rho_k^{x,m}(S_{k+1}^{T-n}) \circ \tilde{\theta}_n \\
			&= \rho_{k+n}^{x,m}(g(M_{T} - X_{k+n})) \wedge \rho_{k+n}^{x,m}(S_{k+1}^{T-n} \circ \tilde{\theta}_n) \\
			&= \rho_{k+n}^{x,m}(g(M_{T} - X_{k+n})) \wedge \rho_{k+n}^{x,m}(S_{k+n+1}^T) \\
			&= S_{k+n}^T.
		\end{align*}
		
		\textbf{Step 2.} Let $\tau_n^T := \inf \{k=n,\ldots, T : S_k^T = \rho_k^{x,m}(g(M_T-X_k)) \}$. We show that $\tau_n^T = n + \tau_0^{T-n} \circ \tilde{\theta}_n$.
		
		Indeed, 
		\begin{align*}
			\tau_n^T
			&= \inf \{k=n,\ldots, T : S_{k-n}^{T-n} \circ \tilde{\theta}_n 
			= \rho_{k-n}^{x,m}(g(M_{T-n}-X_{k-n})) \circ \tilde{\theta}_n \} \\
			&= n + \inf \{k=0,\ldots, T-n : S_{k}^{T-n} \circ \tilde{\theta}_n = \rho_{k}^{x,m}(g(M_{T-n}-X_{k})) \circ \tilde{\theta}_n \} \\
			&= n + \tau_0^{T-n} \circ \tilde{\theta}_n.
		\end{align*}
		
		\textbf{Step 3.} We show that for $n=T, \ldots, 0$ we have
		\begin{equation}
			\label{optimality2}
			S_n^T = \rho_n^{x,m}(g(M_T-X_{\tau_n^T}))
		\end{equation}
		Note that on $\{\tau_{n-1}^T \geq n\}$ we have $\tau_{n-1}^T = \tau_n^T$ (by definition of these stopping times).
		From this and time consistency we have
		\begin{multline*}
			\rho_{n-1}^{x,m}(g(M_T-X_{\tau_{n-1}^T})) \\
			= \1_{\{\tau_{n-1}^T=n-1\}} \rho_{n-1}^{x,m}(g(M_T-X_{n-1})) + \1_{\{\tau_{n-1}^T \geq n\}} \rho_{n-1}^{x,m}(\rho_n^{x,m}(g(M_T-X_{\tau_n^T})).
		\end{multline*}
		By the induction hypothesis
		\begin{equation}\label{intermediate_step5}
			\begin{split}
				\rho_{n-1}^{x,m}(g(M_T-X_{\tau_{n-1}^T}))
				= {} & \1_{\{\tau_{n-1}^T=n-1\}} \rho_{n-1}^{x,m}(g(M_T-X_{n-1})) \\
				& + \1_{\{\tau_{n-1}^T \geq n\}} \rho_{n-1}^{x,m}(S_n^T).
			\end{split}
		\end{equation}
		Note that
		\begin{align*}
			S_{n-1}^T &= \rho_{n-1}^{x,m}(g(M_T-X_{n-1})) \qquad \text{ on } \qquad \{ \tau_{n-1}^T = n-1 \}, \\
			S_{n-1}^T &= \rho_{n-1}^{x,m}(S_n^T) \qquad \text{ on } \qquad \{ \tau_{n-1}^T \geq n \}.
		\end{align*}
		Thus, \eqref{intermediate_step5} implies that $\rho_{n-1}^{x,m}(g(M_T-X_{\tau_{k-1}^T})) = S_{n-1}^T$.
		
		\textbf{Step 4.}
		We prove that
		\begin{equation}
			\label{claim_of_step_4}
			S_n^T = \tilde{V}^{T-n}(X_n,M_n).
		\end{equation} 
		
		We have
		\begin{multline}
			\label{equality_in_step_4}
			S_n^T
			= \rho_n^{x,m}(g(M_T-X_{\tau_n^T}))
			= \rho_n^{x,m}(g(M_T-X_{n + \tau_0^{T-n} \circ \tilde{\theta}_n})) \\
			= \rho_n^{x,m}(g(M_{T-n}-X_{\tau_0^{T-n}}) \circ \tilde{\theta}_n)
			= \rho^{X_n,M_n}(g(M_{T-n}-X_{\tau_0^{T-n}})).
		\end{multline}
		On the other hand, one can show by induction that for each $k=T, \ldots, 0$ and
		every $\tau\in \mathscr{T}_{[k,T]}$ we have $\rho_k^{x,m}(g(M_T-X_\tau)) \geq S_k^T$. The claim is true for $k=T$, and we may write 
		\begin{align*}
			\rho^{x,m}_{k-1}(g(M_T-X_\tau)) 
			&= \1_{\{\tau=k-1\}} \rho_{k-1}^{x,m}(g(M_T-X_\tau)) \\ &\quad + \1_{\{\tau \geq k\}} \rho_{k-1}^{x,m}(\rho_k^{x,m}(g(M_T-X_{\tau\vee k}))).
		\end{align*}
		By the induction hypothesis, since $\tau \vee k \in \mathscr{T}_{[k,T]}$ we have
		\begin{align*}
			\rho^{x,m}_{k-1}(g(M_T-X_\tau)) 
			&\geq \1_{\{\tau=k-1\}} \rho_{k-1}^{x,m}(g(M_T-X_k)) + \1_{\{\tau \geq k\}} \rho_{k-1}^{x,m}(S_k^T) \\
			&\geq  \1_{\{\tau=k-1\}} S_{k-1}^T + \1_{\{\tau \geq k\}} S_{k-1}^T \\
			&=S_{k-1}^T.
		\end{align*}
		In particular for $k=0$ we conclude by Step 3 that for every $T \in \mathbb{N}_0$, the stopping time $\tau_0^T$ is optimal and $\tilde{V}^T(x,m)=\rho^{x,m}(g(M_T-X_{\tau_0^T}))$. Combining this with \eqref{equality_in_step_4} gives \eqref{claim_of_step_4}.
		
		\textbf{Step 5.}
		We have by the previous step and the Markov property that
		\begin{align*}
			\tilde{V}^{T-n}(X_n,M_n) 
			&= S_n^T \\
			&= \rho_n^{x,m}(g(M_T-X_n)) \wedge \rho_n^{x,m}(S_{n+1}^T) \\
			&= \rho_n^{x,m}(g(M_N-X_n)) \wedge \rho_n^{x,m}(\tilde{V}^{T-n-1}(X_{n+1},M_{n+1})) \\
			&= \rho_n^{x,m}(g(M_T-X_n)) \wedge \rho^{X_n,M_n}(\tilde{V}^{T-n-1}(X_1,M_1)). 
		\end{align*}
		Taking $n=0$ we get $\tilde{V}^{T}(x,m) = \rho^{x,m}(g(M_T-x)) \wedge \rho^{x,m}(\tilde{V}^{T-1}(X_1,M_1))$, and the result follows by construction.
	\end{proof}
	
	\begin{appendix}
		\section{Monotone Class Theorem}
		\label{sec:MCT}
		For the reader's convenience we state the monotone class theorem in the form given in Th.\ 2.12.9 \cite{Bogachev_I}.
		
		\begin{theorem}
			\label{th_from_Bogachev}
			Let $\mathscr{H}$ be a class of real functions on a set $\Omega$ such that $1 \in \mathscr{H}$ and let $\mathscr{H}_0$ be a subset in $\mathscr{H}$. Then, any of the following conditions yields that $\mathscr{H}$ contains all bounded functions measurable with respect to the $\sigma$-algebra generated by $\mathscr{H}_0$:
			\begin{enumerate}
				\item[(i)] $\mathscr{H}$ is a closed linear subspace in the space of all bounded functions on $\Omega$ with the norm $\vert f \vert := \sup_\Omega \vert f(\omega) \vert$ such that $\lim_{n\to\infty} f_n \in \mathscr{H}$ for every increasing uniformly bounded sequence of nonnegative functions $f_n \in \mathscr{H}$, and, in addition, $\mathscr{H}_0$ is closed with respect to multiplication (i.e., $f g \in \mathscr{H}_0$ for all functions $f, g \in \mathscr{H}_0$).
				\item[(ii)] $\mathscr{H}$ is closed with respect to the formation of uniform limits and monotone limits and $\mathscr{H}_0$ is an algebra of functions (i.e., $f + g$, $cf$, $f g \in \mathscr{H}_0$ for all $f, g \in \mathscr{H}_0$, $c \in \mathbb{R}$) and $1 \in \mathscr{H}_0$.
				\item[(iii)] $\mathscr{H}$ is closed with respect to monotone limits and $\mathscr{H}_0$ is a linear space containing $1$ such that $\min(f, g) \in \mathscr{H}_0$ for all $f, g \in \mathscr{H}_0$.
			\end{enumerate}
		\end{theorem}
	\end{appendix}

	
\providecommand{\etalchar}[1]{$^{#1}$}

\end{document}